\documentclass[journal]{IEEEtran}
\usepackage{cite}
\usepackage{amsmath}
\usepackage{amssymb}
\usepackage{amsthm}
\usepackage{arydshln}
\usepackage{mathtools}
\usepackage{mathrsfs}
\usepackage{algorithmic}
\usepackage{graphicx}
\usepackage{epsfig,graphics}
\usepackage{epstopdf}
\usepackage{subfigure}
\usepackage{multirow,ragged2e}
\usepackage{bm}
\usepackage{enumitem}
\usepackage{leftidx}
\usepackage{filecontents}
\graphicspath{{C:/Users/TOSHIBA/Desktop/Figures/}}

\theoremstyle{theorem}
\newtheorem*{theorem}{Theorem}
\theoremstyle{lemma}

\theoremstyle{assumption}

\theoremstyle{problem}
\newtheorem*{problem}{Problem}
\theoremstyle{proposition}
\newtheorem*{proposition}{Proposition}

\newcommand{\fig}[1]{Fig.~\ref{#1}}

\newcommand{\eq}[1]{Equation~(\ref{#1})}

\newcommand{\fbm}[1]{\mathbf{#1}}
\newcommand{\tbm}[1]{\fbm{#1}^\mathsf{T}}
\newcommand{\tfbm}[1]{\bm{#1}^\mathsf{T}}

\newcommand{\dottbm}[1]{\dot{\fbm{#1}}^\mathsf{T}}
\newcommand{\dotbm}[1]{\dot{\fbm{#1}}}

\newcommand{\ddotbm}[1]{\ddot{\fbm{#1}}}
\newcommand{\dddotbm}[1]{\dddot{\fbm{#1}}}

\begin{document}

\title{Connectivity-Preserving Coordination Control of Multi-Agent Systems with Time-Varying Delays}

\author{Yuan~Yang,~\IEEEmembership{Student Member,~IEEE,} Yang Shi,~\IEEEmembership{Fellow,~IEEE} and Daniela Constantinescu,~\IEEEmembership{Member,~IEEE,}\thanks{The authors are with the Department of Mechanical Engineering, University of Victoria, Victoria, BC V8W 2Y2 Canada (e-mail: yangyuan@uvic.ca; danielac@uvic.ca; yshi@uvic.ca).}} 
       
\maketitle

\begin{abstract}
This paper presents a distributed position synchronization strategy that also preserves the initial communication links for single-integrator multi-agent systems with time-varying delays. The strategy employs a coordinating proportional control derived from a specific type of potential energy, augmented with damping injected through a dynamic filter. The injected damping maintains all agents within the communication distances of their neighbours, and asymptotically stabilizes the multi-agent system, in the presence of time delays. Regarding the closed-loop single-integrator multi-agent system as a double-integrator system suggests an extension of the proposed strategy to connectivity-preserving coordination of Euler-Lagrange networks with time-varying delays. Lyapunov stability analysis and simulation results validate the two designs.
\end{abstract}


\section{Introduction}\label{Sec: introduction}

Distributed coordination control of multi-agent systems~(MAS-s) has received considerable attention in the past decade~\cite{Ren2013TIE,Shi2017TIE}. Simple Proportional (P) and Proportional-Derivative~(PD) coordination strategies have been proven effective for networks of single-integrator~\cite{Murray2004TAC}, double-integrator~\cite{Ren2008TAC}, Euler-Lagrange~(EL)~\cite{Ren2009IJC}, and more general second-order nonlinear~\cite{Ren2013Automatica} systems. Further, adaptive control and output feedback control have overcome system uncertainties~\cite{Ren2013Automatica2,Wang2013TAC} and the lack of velocity measurements~\cite{Duan2014TAC,Jiang2017TAC}, respectively.\par

Time-delayed communications foil agents' immediate access to their neighbour information and, thus, distort the inter-agent couplings and endanger synchronization. Prior work has mitigated time delays in various ways. Rendering the MAS input-output passive has led to output feedback synchronization of affine nonlinear networks with constant time delays in~\cite{Chopra2012TAC}. Sufficient damping injection has provably coordinated EL networks with time-varying delays in~\cite{Nuno2013TRO}. An Immersion and Invariance~($\text{I\&I}$) observer has facilitated globally exponentially convergent velocity estimation and output feedback consensus of such networks in~\cite{Nuno2016TCNS}. The need for accurate system dynamics in the $\text{I\&I}$ observer design has been eliminated in~\cite{Nuno2017TCST}. Adaptive consensus of uncertain EL networks has been explored in~\cite{Nuno2011TAC,Wang2014TAC} for constant time delays, and in~\cite{Polushin2014TAC,Dixon2017Cybernetics,Jia2017Cybernetics} for time-varying delays with sufficiently small rates of change of the delays.

Conventional MAS coordination strategies assume a connected communication graph. In contrast, gradient-based controls derived from potentials dependent on inter-agent distances can both synchronize and maintain the connectivity of a MAS. Existing research has offered several potentials for connectivity-preserving MAS synchronization. Unbounded controls, derived from potentials that continuously increase inter-agent attraction when adjacent systems approach their communication boundaries, have prevented disconnection among single-integrator and nonholonomic agents in~\cite{Egerstedt2007TRO,Dimos2007TAC,Dimos2008TRO,Zavlanos2008TRO}. Bounded controls have achieved similar objectives in~\cite{Dimos2008ICRA,Amir2010TAC,Feng2013Automatica,Sun2013Automatica,Amir2013TAC}. A specific class of bounded potentials has led to PD-like, as well as adaptive PD-like, control with distance-dependent P gains for double-integrator networks in~\cite{Su2010SCL}, and for second-order nonlinear networks in~\cite{Su2011Automatica}, respectively. 

Connectivity-preserving coordination of MAS-s with time-varying delays has been less explored to date, and is the focus herein. The paper proposes a class of potential functions of inter-agent distances whose negative gradients it adopts as control inputs for single-integrator MAS-s. In the presence of communication delays, agents cannot access timely neighbour information and gradient-based controls suffer from disturbances that would lead to instability. A dynamic filter, that increases the degree of single-integrator MAS-s from one to two, suppresses the delay-induced disturbances through sufficient damping injection. The injected damping limits velocities appropriately to maintain each agent in its neighbours' communication zones although the neighbours move during the time it takes their information to reach the agent. Because the dynamic filter leads to a closed-loop system that can be regarded as a double-integrator MAS, the proposed methodology is readily extended to connectivity-preserving coordination of EL networks with time-varying delays. Rigorous Lyapunov stability analysis and simulations validate the proposed designs.     
\section{Preliminaries}\label{Sec: preliminaries}

Let a MAS have $N$ agents, and let each agent $i$ have a broadcasting area~(communication zone) of radius $r$ centered around its position $\fbm{x}_{i}$. At time $t$, agent $i$ can receive the delayed position $\fbm{x}_{jd}=\fbm{x}_{j}(t-d_{ji}(t))$ of agent $j$ only if agent $j$ is in the current communication zone of agent $i$, i.e., $|\fbm{x}^{d}_{ij}|=|\fbm{x}_{i}(t)-\fbm{x}_{j}(t-d_{ji}(t))|=\sqrt{{\fbm{x}^{d}_{ij}}^\mathsf{T}\cdot  \fbm{x}^{d}_{ij}}<r$, with $0\leq d_{ji}(t)\leq \overline{d}_{ji}$ the time-varying delay from $j$ to $i$.\par 

The communications of the MAS are described by its communication graph $\mathcal{G}$, that comprises two sets $\mathcal{G}=\{\mathcal{V},\mathcal{E}\}$: the set of graph nodes $\mathcal{V}=\{1,\cdots,N\}$, that includes all agents in the MAS; the set of graph edges $\mathcal{E}=\{(i,j)\in\mathcal{V}\times\mathcal{V}\}$, that collects all communication links $(i,j)$ in the system. An oriented communication link or graph edge $(i,j)$ exists if agent $j$ receives information from agent $i$. Agent $i$ is a neighbour of agent $j$, i.e., $i\in\mathcal{N}_{j}$, if and only if $(i,j)\in\mathcal{E}$, where $\mathcal{N}_{j}$ is the set of all neighbours of agent $j$. A graph $\mathcal{G}$ is undirected if agents $i$ and $j$ are adjacent, or neighbours of each other, i.e., $i\in\mathcal{N}_{j}$ if and only if $j\in\mathcal{N}_{i}$. A sequence of adjacent agents, $(v_{1},v_{2}), (v_{2},v_{3}), \cdots, (v_{m-1},v_{m})$ forms a path in an undirected graph. Then, an undirected graph is connected if and only if there exists a path between each pair of agents in the system.

The adjacency matrix $\fbm{A}_{N\times N}=[a_{ij}]$ associated to an undirected graph $\mathcal{G}$ is a symmetric matrix, with $a_{ij}>0$ if $(i,j)\in\mathcal{E}$ and $a_{ij}=0$ otherwise. The corresponding weighted Laplacian $\fbm{L}_{N\times N}=[l_{ij}]$ is also symmetric, with $l_{ii}=\sum_{k\in\mathcal{N}_{i}}a_{ik}$ and $l_{ij}=-a_{ij}$ for $j\neq i$. Assume that the undirected communication graph $\mathcal{G}$ contains $2M$ oriented edges. For each pair of adjacent agents $i$ and $j$, only one of the oriented links $(i,j)$ and $(j,i)$ is labelled as $e_{k}$, $k=1,\cdots,M$, with weight $w(e_{k})=a_{ij}=a_{ji}$. Particularly, if $e_{k}=(i,j)$, then agents $i$ and $j$ are the tail and the head of $e_{k}$. The incidence matrix $\fbm{D}_{N\times M}=[\delta_{hk}]$ of the graph is defined by $\delta_{hk}=1$ if agent $h$ is the head of $e_{k}$, $\delta_{hk}=-1$ if agent $h$ is the tail of $e_{k}$, and $\delta_{hk}=0$ otherwise. The graph Laplacian and the incidence matrix obey~\cite{Egerstedt2010Princeton}: $\fbm{L}=\fbm{D}\fbm{W}\tbm{D}$, where $\fbm{W}$ is a $M\times M$ diagonal matrix with $w(e_{k})$ on its diagonal.\par

The following assumption underlies the connectivity-preserving synchronization strategy designed in this paper:
\begin{enumerate}[label=A.\arabic*]
\item \label{A1}
The undirected communication graph of the MAS is initially connected and each pair of adjacent agents $(i,j)$ is strictly within their communication distance, i.e., $|\fbm{x}_{ij}(0)|=|\fbm{x}_{i}(0)-\fbm{x}_{j}(0)|\leq r-\epsilon$ for some $\epsilon>0$.
\end{enumerate}
Under assumption~\ref{A1}, the connectivity of an MAS with static communication graph is preserved if all initial communication links are maintained. In an MAS with dynamic graph, initially non-adjacent pairs of agents can start to receive each other's delayed information once they enter in each other's communication zone. A dynamic graph requires a mechanism to create new links depending on inter-agent distances and to simultaneously maintain connectivity. When communication delays prevent agents' access to timely information of other agents, building bidirectional links dynamically is not trivial. For simplicity, this paper considers only MAS-s with static communication graph and aims to render the initial edge set invariant~(static) through control. Thus, the connectivity-preserving coordination problem in this paper is the following:

\begin{problem}
Let a MAS with time-varying communication delays have $N$ agents with broadcasting radius $r$ and obey assumption~\ref{A1}. Design distributed control laws that use only the agents' information and their neighbours' delayed positions to render the original edge set $\mathcal{E}(0)$ invariant and coordinate the MAS.
\end{problem}  

In addition to assumption~\ref{A1}, the stability analysis in the next section uses the following lemma~\cite{Nuno2009IJRR}:
\begin{enumerate}[label=L.\arabic*]
\item \label{L1}
For any vectors $\fbm{x}, \fbm{y}\in\mathbb{R}^{n}$, any variable time delay bounded by $0\leq d(t)\leq \overline{d}<\infty$ and any constant $\alpha>0$, the following inequality holds
\begin{align*}
2\int^{t}_{0}\tbm{x}(\sigma)\int^{0}_{-d(\sigma)}\fbm{y}(\sigma+\theta)d\theta d\sigma\leq \alpha\|\fbm{x}\|^{2}_{2}+\frac{\overline{d}^{2}}{\alpha}\|\fbm{y}\|^{2}_{2}\textrm{,}
\end{align*}
where $\|\fbm{x}\|_{2}$ is the $\mathcal{L}_{2}$-norm of $\fbm{x}$, defined by $\|\fbm{x}\|_{2}^{2}=\int^{t}_{0}\tbm{x}(\sigma)\fbm{x}(\sigma)d\sigma$.
\end{enumerate}

\section{Connectivity-Preserving Coordination}\label{Sec: connectivity-preserving coordination}

The dynamics of a MAS with $N$ single-integrator $n$-dimensional agents are:
\begin{equation}\label{equ1}
\dotbm{x}_{i}=\fbm{u}_{i}\textrm{,}
\end{equation}
where: $i\in {1,\cdots,N}$ indexes the agent; and $\fbm{x}_{i}$, $\dotbm{x}_{i}$ and $\fbm{u}_{i}$ are the agent's position, velocity and control input. The potential $\psi(|\fbm{x}_{ij}|)$ used in this paper to quantify the energy stored in the inter-agent connections belongs to a class of functions with the following properties:
\begin{enumerate}
\item[1.]\label{P1}
$\psi(|\fbm{x}_{ij}|)\in C^{1}$ is positive definite and strictly increasing with respect to $|\fbm{x}_{ij}|$ for $|\fbm{x}_{ij}|\in [0,r]$.
\item[2.]\label{P2}
$\frac{N(N-1)}{2}\psi(r-\epsilon)<\psi(r)$.
\item[3.]\label{P3}
$\boldsymbol{\zeta}(\nabla_{i}\psi(|\fbm{x}_{ij}|)-\nabla_{i}\psi(|\fbm{x}^{d}_{ij}|))\leq \gamma\boldsymbol{\zeta}(\fbm{x}_{j}-\fbm{x}_{jd})+\eta\bar{x}_{j}\fbm{1}$ for bounded $\boldsymbol{\zeta}(\fbm{x}_{ij})$ and $\boldsymbol{\zeta}(\fbm{x}^{d}_{ij})$, where: $\boldsymbol{\zeta}(\cdot):\mathbb{R}^{n}\to \mathbb{R}^{n}$ maps the components of a vector to their absolute values $\boldsymbol{\zeta}(\fbm{x})=(|x_1|\cdots |x_n|)^\mathsf{T}$; $\nabla_{i}\psi(\cdot)$ are the gradients of $\psi(\cdot)$ with respect to $\fbm{x}_{i}$, $\bar{x}_{j}=\|\fbm{x}_{j}-\fbm{x}_{jd}\|_{\infty}$; $\gamma$ and $\eta$ are positive constants; and $\fbm{1}=[1,\cdots,1]^\mathsf{T}$.
\item[4.]\label{P4}
$\nabla_{i}\psi(|\fbm{x}_{ij}|)=h(|\fbm{x}_{ij}|)(\fbm{x}_{i}-\fbm{x}_{j})$, where $h(|\fbm{x}_{ij}|)$ is increasing with respect to $|\fbm{x}_{ij}|$ and is positively lower-bounded for $|\fbm{x}_{ij}|\in[0,r]$.
\end{enumerate}

The particular potential with the above properties selected in this paper is $\psi(|\fbm{x}_{ij}|)=\frac{|\fbm{x}_{ij}|^{2}}{r^{2}-|\fbm{x}_{ij}|^{2}+Q}$, with $Q>0$ a parameter to be determined, and provides the following control laws:
\begin{equation}\label{equ2}
\dotbm{u}_{i}=-p\sum_{j\in\mathcal{N}_{i}}\nabla_{i}\psi(|\fbm{x}^{d}_{ij}|)-\fbm{K}_{i}\fbm{u}_{i}\textrm{,}
\end{equation}
where $\mathcal{N}_{i}$ is the set of neighbours of agent $i$ at time $t$; $\nabla_{i}\psi(|\fbm{x}^{d}_{ij}|)$ are the gradients of $\psi(|\fbm{x}^{d}_{ij}|)$ with respect to $\fbm{x}_{i}$; and $\fbm{K}_{i}=\text{diag}\{k_{ik}\}$ with $k=1,\cdots,n$ are positive diagonal gain matrices to be determined. The gradient of the selected potential with respect to $\fbm{x}_{i}$ being:
\begin{align*}
\nabla_{i}\psi(|\fbm{x}^{d}_{ij}|)=\frac{2(r^{2}+Q)}{(r^{2}-|\fbm{x}^{d}_{ij}|^{2}+Q)^{2}}\fbm{x}^{d}_{ij}\textrm{,}
\end{align*}    
the designed controls~\eqref{equ2} are proportional controls whose state-dependent gains $h(|\fbm{x}^{d}_{ij}|)=\frac{2(r^{2}+Q)}{(r^{2}-|\fbm{x}^{d}_{ij}|^{2}+Q)^{2}}$ increase with $|\fbm{x}^{d}_{ij}|$ on $[0,\sqrt{r^{2}+Q})$. In other words, the controls~\eqref{equ2} stiffen the inter-agent couplings when the connections are threatened. 

A property of the selected potential, needed in the connectivity preservation and coordination proofs, is the following:
\begin{proposition}
For the function $\psi(|\fbm{x}_{ij}|)=\frac{|\fbm{x}_{ij}|^{2}}{r^{2}-|\fbm{x}_{ij}|^{2}+Q}$, proper selection of $Q$ guarantees that $\frac{N(N-1)}{2}\psi(r-\epsilon)<\psi(r)$.
\end{proposition}
\begin{proof}
$\frac{N(N-1)}{2}\psi(r-\epsilon)=\frac{N(N-1)}{2}\frac{(r-\epsilon)^{2}}{r^{2}-(r-\epsilon)^{2}+Q}$ and $\psi(r)=\frac{r^{2}}{Q}$, so $\frac{N(N-1)}{2}\psi(r-\epsilon)<\psi(r)$ implies that $\frac{N(N-1)}{2}\frac{(r-\epsilon)^{2}}{r^{2}-(r-\epsilon)^{2}+Q}<\frac{r^{2}}{Q}$, i.e., $N(N-1)(r-\epsilon)^{2}Q<2r^{2}[r^{2}-(r-\epsilon)^{2}+Q]$, and further $[N(N-1)(r-\epsilon)^{2}-2r^{2}]Q<2r^{2}[r^{2}-(r-\epsilon)^{2}]$. This can be guaranteed by selecting $Q<\frac{2r^{2}[r^{2}-(r-\epsilon)^{2}]}{N(N-1)(r-\epsilon)^{2}-2r^{2}}$ if $N(N-1)>\frac{2r^{2}}{(r-\epsilon)^{2}}$; and $Q>0$ arbitrarily otherwise.  
\end{proof}

A property of time-delayed signals needed in the proofs is:
\begin{proposition}
Given any vectors $\fbm{x}\in\mathbb{R}^{n}$, $\fbm{y}\in\mathbb{R}^{n}$ and any variable time delays $d(t)$ bounded by $0\leq d(t)\leq \overline{d}$, they obey
\begin{align*}
\int^{t}_{0}\tbm{x}(\sigma)\bar{y}(\sigma)\fbm{1}d\sigma\leq\frac{\alpha}{2}\|\fbm{x}\|^{2}_{2}+\frac{n\overline{d}^{2}}{2\alpha}\|\fbm{y}\|^{2}_{2}\textrm{,}
\end{align*}
for some $\alpha>0$ and $\bar{y}(\sigma)=\max\limits_{n}\boldsymbol{\zeta}\left(\int^{\sigma}_{\sigma-d(\sigma)}\fbm{y}(\theta)d\theta\right)$.
\end{proposition}
\begin{proof}
The Cauchy-Schwartz inequality yields 
\begin{align*}
&\int^{t}_{0}\tbm{x}(\sigma)\bar{y}(\sigma)\fbm{1}d\sigma\leq\left(\int^{t}_{0}|\fbm{x}(\sigma)|^{2}d\sigma\right)^{\frac{1}{2}}\left(\int^{t}_{0}n\bar{y}(\sigma)^{2}d\sigma\right)^{\frac{1}{2}}\\
&\leq\frac{\alpha}{2}\|\fbm{x}\|^{2}_{2}+\frac{n}{2\alpha}\int^{t}_{0}\bar{y}(\sigma)^{2}d\sigma\textrm{,}
\end{align*}
where $|\fbm{x}|^{2}=\tbm{x}\fbm{x}$. Then, $\bar{y}(\sigma)^{2}$ can be upper-bounded by
\begin{align*}
\bar{y}(\sigma)^{2}\leq\Big|\int^{\sigma}_{\sigma-d(\sigma)}\fbm{y}(\theta)d\theta\Big|^{2}\leq d(\sigma)\int^{\sigma}_{\sigma-d(\sigma)}|\fbm{y}(\theta)|^{2}d\theta\textrm{,}
\end{align*}
after applying the Cauchy-Schwartz inequality again. Lastly, using $d(\sigma)\leq\overline{d}$ and inverting the integration order leads to
\begin{align*}
&\int^{t}_{0}\tbm{x}(\sigma)\bar{y}(\sigma)\fbm{1}d\sigma\leq\frac{\alpha}{2}\|\fbm{x}\|^{2}_{2}+\frac{n\overline{d}}{2\alpha}\int^{t}_{0}\int^{\sigma}_{\sigma-d(\sigma)}|\fbm{y}(\theta)|^{2}d\theta d\sigma\\
&\leq\frac{\alpha}{2}\|\fbm{x}\|^{2}_{2}+\frac{n\overline{d}}{2\alpha}\int^{t}_{0}\int^{0}_{-d(\sigma)}|\fbm{y}(\sigma+\theta)|^{2}d\theta d\sigma\\
&\leq\frac{\alpha}{2}\|\fbm{x}\|^{2}_{2}+\frac{n\overline{d}}{2\alpha}\int^{t}_{0}\int^{0}_{-\overline{d}}|\fbm{y}(\sigma+\theta)|^{2}d\theta d\sigma\\
&=\frac{\alpha}{2}\|\fbm{x}\|^{2}_{2}+\frac{n\overline{d}}{2\alpha}\int^{0}_{-\overline{d}}\int^{t}_{0}|\fbm{y}(\sigma+\theta)|^{2}d\sigma d\theta\\
&\leq\frac{\alpha}{2}\|\fbm{x}\|^{2}_{2}+\frac{n\overline{d}}{2\alpha}\int^{0}_{-\overline{d}}\|\fbm{y}\|^{2}_{2}d\theta=\frac{\alpha}{2}\|\fbm{x}\|^{2}_{2}+\frac{n\overline{d}^{2}}{2\alpha}\|\fbm{y}\|^{2}_{2}\textrm{.}
\end{align*}
\end{proof}

The Lyapunov function that quantifies the energy of the single-integrator MAS in closed loop with the controls~\eqref{equ2} is:
\begin{equation}\label{equ3}
V=\frac{p}{2}\sum^{N}_{i=1}\sum_{j\in\mathcal{N}_{i}(0)}\psi(|\fbm{x}_{ij}|)+\frac{1}{2}\sum^{N}_{i=1}\tbm{u}_{i}\fbm{u}_{i}\textrm{.}
\end{equation}
Assumption~\ref{A1} together with the first property of $\psi(\cdot)$ imply that $V(0)\geq 0$. The selection $\fbm{u}_{i}(0)=\fbm{0}$ and the second property of $\psi(\cdot)$ lead to $V(0)\leq\frac{N(N-1)p}{2}\psi(r-\epsilon)<p\psi(r)$.  Then, because $V$ is continuous, rendering it non-increasing is sufficient for connectivity preservation, i.e.,  $V(t)\leq V(0)<p\psi(r) \Rightarrow$ $\psi(|\fbm{x}_{ij}(t)|)<\psi(r)$ for all $(i,j)\in\mathcal{E}(0)$ and $t>0$. If all initial connections are maintained, then $\psi(|\fbm{x}_{ij}|)\geq 0$ for all $(i,j)\in\mathcal{E}(0)$ and $t>0$, and $V$ is positive definite with respect to $\fbm{x}_{ij}$ and $\fbm{u}_{i}$. Then, coordination can be shown by showing that $V\to 0$ as $t\to \infty$.  

By continuity of $V$, there exists $t_{1}>0$ such that $\mathcal{N}_{i}=\mathcal{N}_{i}(0)$ on $[0,t_{1}]$, $\forall i=1,\cdots,N$. The derivative of $V$ along the closed-loop dynamics \eqref{equ1} and \eqref{equ2} on $[0,t_{1}]$ are:
\begin{align*}
\dot{V}=&\frac{p}{2}\sum^{N}_{i=1}\sum_{j\in\mathcal{N}_{i}(0)}\left[\dottbm{x}_{i}\nabla_{i}\psi(|\fbm{x}_{ij}|)+\dottbm{x}_{j}\nabla_{j}\psi(|\fbm{x}_{ij}|)\right]+\sum^{N}_{i=1}\tbm{u}_{i}\dotbm{u}_{i}\\
=&\frac{p}{2}\left(\sum^{N}_{i=1}\sum_{j\in\mathcal{N}_{i}(0)}\dottbm{x}_{i}\nabla_{i}\psi(|\fbm{x}_{ij}|)+\sum^{N}_{j=1}\sum_{i\in\mathcal{N}_{j}(0)}\dottbm{x}_{j}\nabla_{j}\psi(|\fbm{x}_{ij}|)\right)\\
&+\sum^{N}_{i=1}\tbm{u}_{i}\left(-p\sum_{j\in\mathcal{N}_{i}}\nabla_{i}\psi(|\fbm{x}^{d}_{ij}|)-\fbm{K}_{i}\fbm{u}_{i}\right)\\
=&p\sum^{N}_{i=1}\sum_{j\in\mathcal{N}_{i}(0)}\dottbm{x}_{i}\left[\nabla_{i}\psi(|\fbm{x}_{ij}|)-\nabla_{i}\psi(|\fbm{x}^{d}_{ij}|)\right]-\sum^{N}_{i=1}\dottbm{x}_{i}\fbm{K}_{i}\dotbm{x}_{i}\textrm{,}
\end{align*}
where the system dynamics $\dotbm{x}_{i}=\fbm{u}_{i}$ and the initial undirected and connected interaction graph have been applied. From the third property of the potential $\psi(\cdot)$, it follows that $\boldsymbol{\zeta}\left(\nabla_{i}\psi(|\fbm{x}_{ij}|)-\nabla_{i}\psi(|\fbm{x}^{d}_{ij}|)\right)\leq\gamma\boldsymbol{\zeta}\left(\fbm{x}_{j}-\fbm{x}_{jd}\right)+\eta\bar{x}_{j}\fbm{1}\leq \gamma\int^{t}_{t-d_{ji}(t)}|\dotbm{x}_{j}(\theta)|d\theta+\eta\bar{x}_{j}\fbm{1}$ and $\dot{V}$ can be upper-bounded by:
\begin{align*}
\dot{V}\leq&p\sum^{N}_{i=1}\sum_{j\in\mathcal{N}_{i}(0)}\gamma\boldsymbol{\zeta}\left(\dotbm{x}_{i}\right)^\mathsf{T}\int^{t}_{t-d_{ji}(t)}\boldsymbol{\zeta}\left(\dotbm{x}_{j}(\theta)\right)d\theta\\
&+p\sum^{N}_{i=1}\sum_{j\in\mathcal{N}_{i}(0)}\eta\bar{x}_{j}\boldsymbol{\zeta}\left(\dotbm{x}_{i}\right)^\mathsf{T}\fbm{1}-\sum^{N}_{i=1}\dottbm{x}_{i}\fbm{K}_{i}\dotbm{x}_{i}\textrm{.}
\end{align*}
Lemma~\ref{L1} and time integration of the above inequality yield
\begin{equation}\label{equ4}
\begin{aligned}
V(t)\leq &p\gamma\sum^{N}_{i=1}\sum_{j\in\mathcal{N}_{i}(0)}\int^{t}_{0}\boldsymbol{\zeta}\left(\dotbm{x}_{i}(\sigma)\right)^\mathsf{T}\int^{\sigma}_{\sigma-d_{ji}(\sigma)}\boldsymbol{\zeta}\left(\dotbm{x}_{j}(\theta)\right)d\theta d\sigma\\
&+p\eta\sum^{N}_{i=1}\sum_{j\in\mathcal{N}_{i}(0)}\int^{t}_{0}\boldsymbol{\zeta}\left(\dotbm{x}_{i}(\sigma)\right)^\mathsf{T}\bar{x}_{j}(\sigma)\fbm{1}d\sigma\\
&-\sum^{N}_{i=1}\int^{t}_{0}\dottbm{x}_{i}(\sigma)\fbm{K}_{i}\dotbm{x}_{i}(\sigma)d\sigma+V(0)\\
\leq&-\sum^{N}_{i=1}k_{i}\|\dotbm{x}_{i}\|^{2}_{2}+\sum^{N}_{i=1}\sum_{j\in\mathcal{N}_{i}(0)}\Bigg(\frac{\alpha_{ij}p(\gamma+\eta)}{2}\|\dotbm{x}_{i}\|^{2}_{2}\\
&\ \ +\frac{p(\gamma+n\eta)\overline{d}^{2}_{ji}}{2\alpha_{ij}}\|\dotbm{x}_{j}\|^{2}_{2}\Bigg)+V(0)\textrm{,}
\end{aligned}
\end{equation}
with $\alpha_{ij}$ positive constants and $k_{i}$ the smallest eigenvalues of $\fbm{K}_{i}$. 

Letting $\bm{\chi}=[\|\dotbm{x}_{1}\|^{2}_{2},\cdots,\|\dotbm{x}_{N}\|^{2}_{2}]^\mathsf{T}$, \eq{equ4} can further be rewritten as $V(t)\leq V(0)-\tbm{1}\bm{\Phi}\bm{\chi}$, where $\bm{\Phi}_{N\times N}=[\phi_{ij}]$ with $\phi_{ij}=k_{i}-\frac{p(\gamma+\eta)}{2}\sum_{j\in\mathcal{N}_{i}(0)}\alpha_{ij}$ if $j=i$, $\phi_{ij}=-\frac{p(\gamma+n\eta)\overline{d}^{2}_{ji}}{2\alpha_{ij}}$ if $j\neq i$ and $j\in\mathcal{N}_{i}(0)$, and $\phi_{ij}=0$ otherwise. It then follows that $V(t)\leq V(0)$ if all column summations of $\bm{\Phi}$ are positive, i.e.,
\begin{equation}\label{equ5}
k_{i}>\sum_{j\in\mathcal{N}_{i}(0)}\left(\frac{\alpha_{ij}p(\gamma+\eta)}{2}+\frac{p(\gamma+n\eta)\overline{d}^{2}_{ij}}{2\alpha_{ji}}\right)\textrm{.}
\end{equation}
In other words, $p\psi(|\fbm{x}_{ij}|)\leq V(t)\leq V(0)<p\psi(r)$, $\forall (i,j)\in\mathcal{E}(0)$ and connectivity is maintained if the damping gains $\fbm{K}_{i}$ can be selected sufficiently large to satisfy~\eqref{equ5}.

The strict inequality~\eqref{equ5} means that there exists $\bm{\lambda}_{N\times 1}$ with strictly positive components such that $V(0)-V(t)\geq \tfbm{\lambda}\bm{\chi}$, and thus $\dotbm{x}_{i}\in\mathcal{L}_{2}$, $\forall i=1,\cdots,N$. In turn, $V(t)\leq V(0)$ indicates that $\fbm{u}_{i}=\dotbm{x}_{i}\in\mathcal{L}_{\infty}$. Then, the derivative of the closed-loop dynamics~\eqref{equ1} and~\eqref{equ2} imply that $\ddotbm{x}_{i}\in\mathcal{L}_{\infty}$. Using Barbalat's lemma, it follows that $\dotbm{x}_{i}\to\fbm{0}$ as $t\to +\infty$, $\forall i=1,\cdots,N$. The derivative of the controller dynamics~\eqref{equ2} means $\dddotbm{x}_{i}\in\mathcal{L}_{\infty}$, and thus Barbalat's lemma indicates $\sum_{j\in\mathcal{N}^{0}_{i}}\nabla_{i}\psi(|\fbm{x}^{d}_{ij}|)\to\fbm{0}$, i.e., $\sum_{j\in\mathcal{N}_{i}(0)}h(|\fbm{x}^{d}_{ij}|)(\fbm{x}_{i}(t)-\fbm{x}_{j}(t-d_{ji}(t)))\to\fbm{0}$. From $\dotbm{x}_{i}\to\fbm{0}$, it follows that $\fbm{x}^{d}_{ij}=\fbm{x}_{i}(t)-\fbm{x}_{j}(t-d_{ji}(t))=\fbm{x}_{i}(t)-\fbm{x}_{j}(t)+\int^{t}_{t-d_{ji}(t)}\dotbm{x}_{j}(\theta)d\theta\to\fbm{x}_{ij}$. Therefore, $\sum_{j\in\mathcal{N}_{i}(0)}h(|\fbm{x}_{ij}|)(\fbm{x}_{i}-\fbm{x}_{j})\to\fbm{0}$, $\forall i=1,\cdots,N$, i.e., $(\fbm{L}(\fbm{x})\otimes\fbm{I}_{n})\fbm{x}\to\fbm{0}$, where $\fbm{L}(\fbm{x})=[l_{ij}(\fbm{x})]$ with $l_{ij}(\fbm{x})=\sum_{k\in\mathcal{N}_{i}(0)}h(|\fbm{x}_{ik}|)$ for $j=i$, $l_{ij}(\fbm{x})=-h(|\fbm{x}_{ij}|)$ for $j\neq i\ \text{and}\ j\in\mathcal{N}_{i}(0)$, and $l_{ij}(\fbm{x})=0$ for $j\neq i\ \text{and}\ j\notin\mathcal{N}_{i}(0)$.
Then, the interaction graph being undirected and connected, $\fbm{L}(\fbm{x})=\fbm{D}\fbm{W}(\fbm{x})\tbm{D}$ and $(\tbm{D}\otimes\fbm{I}_{n})\fbm{x}\to\fbm{0}$, which implies that $\fbm{x}_{1}\to\cdots\to\fbm{x}_{N}$, i.e., coordination is achieved.

The parameters $\gamma$ and $\eta$ in the design criterion~\eqref{equ5} can be specified for the potential $\psi(|\fbm{x}_{ij}|)=\frac{|\fbm{x}_{ij}|^{2}}{r^{2}-|\fbm{x}_{ij}|^{2}+Q}$ noting that the connectivity preservation and coordination proof shows $V(t^{+})\leq V(0)$ under the condition that $V(t^{-})\leq V(0)$. If there exist $\gamma$ and $\eta$ that satisfy the third property of $\psi(\cdot)$ on $\{(\fbm{x},\dotbm{x})|\ V(t)\leq V(0)<p\psi(r)\}$, then $\fbm{K}_{i}$-s that observe~\eqref{equ5} render the set invariant and maintain MAS connectivity.
\begin{proposition}
For the potential function $\psi(|\fbm{x}_{ij}|)=\frac{|\fbm{x}_{ij}|^{2}}{r^{2}-|\fbm{x}_{ij}|^{2}+Q}$ with $Q\geq 2pn^{2}\overline{d}^{2}_{ji}\psi(r)+2nr\overline{d}_{ji}\sqrt{2p\psi(r)}+\Delta$, where $\Delta>0$, there exist $\gamma\geq\frac{2(r^{2}+Q)}{\Delta^{2}}$ and $\eta\geq\frac{4(r^{2}+Q)^{2}(2r+n\overline{d}_{ji}\sqrt{2p\psi(r)})\sqrt{n}r}{Q^{2}\Delta^{2}}$ that guarantee the third property of $\psi(|\fbm{x}_{ij}|)$ on the set $\{(\fbm{x},\dotbm{x})|\ V(t)\leq V(0)<p\psi(r)\}$. 
\end{proposition}
\begin{proof} On the set $\{(\fbm{x},\dotbm{x})|\ V\leq V(0)\}$, $|{x}_{ij}|<r$ and $|\dot{x}_{i,k}|\leq\sqrt{2p\psi(r)}$ $\forall k=1,\cdots,n$ and $\forall i,j=1,\cdots,N$. Because $|\fbm{x}^{d}_{ij}|=|\fbm{x}_{i}-\fbm{x}_{j}+\fbm{x}_{j}-\fbm{x}_{jd}|\leq|\fbm{x}_{ij}|+|\fbm{x}_{j}-\fbm{x}_{jd}|\leq r+n\overline{d}_{ji}\sqrt{2p\psi(r)}$, $\nabla_{i}\psi(|\fbm{x}_{ij}|)$ and $\nabla_{i}\psi(|\fbm{x}^{d}_{ij}|)$ both exist by selecting $Q\geq 2pn^{2}\overline{d}^{2}_{ji}\psi(r)+2nr\overline{d}_{ji}\sqrt{2p\psi(r)}+\Delta$. Then,  
\begin{align*}
&\nabla_{i}\psi(|\fbm{x}_{ij}|)-\nabla_{i}\psi(|\fbm{x}^{d}_{ij}|)\\
=&2(r^{2}+Q)\left[\frac{\fbm{x}_{i}-\fbm{x}_{j}}{(r^{2}-|\fbm{x}_{ij}|^{2}+Q)^{2}}-\frac{\fbm{x}_{i}-\fbm{x}_{jd}}{(r^{2}-|\fbm{x}^{d}_{ij}|^{2}+Q)^{2}}\right]\\
=&2(r^{2}+Q)\left[\frac{\fbm{x}_{i}-\fbm{x}_{j}}{(r^{2}-|\fbm{x}_{ij}|^{2}+Q)^{2}}-\frac{\fbm{x}_{i}-\fbm{x}_{j}}{(r^{2}-|\fbm{x}^{d}_{ij}|^{2}+Q)^{2}}\right]\\
&-\frac{2(r^{2}+Q)(\fbm{x}_{j}-\fbm{x}_{jd})}{(r^{2}-|\fbm{x}^{d}_{ij}|^{2}+Q)^{2}}\\
\leq&\frac{2(r^{2}+Q)(\fbm{x}_{i}-\fbm{x}_{j})}{(r^{2}-|\fbm{x}_{ij}|^{2}+Q)^{2}(r^{2}-|\fbm{x}^{d}_{ij}|^{2}+Q)^{2}}\Big[(r^{2}-|\fbm{x}^{d}_{ij}|^{2}\\
&+Q)^{2}-(r^{2}-|\fbm{x}_{ij}|^{2}+Q)^{2}\Big]+\gamma\boldsymbol{\zeta}\left(\fbm{x}_{j}-\fbm{x}_{jd}\right)\\
\end{align*}
\begin{align*}
=&\frac{2(r^{2}+Q)\left[2(r^{2}+Q)-|\fbm{x}_{ij}|^{2}-|\fbm{x}^{d}_{ij}|^{2}\right]}{(r^{2}-|\fbm{x}_{ij}|^{2}+Q)^{2}(r^{2}-|\fbm{x}^{d}_{ij}|^{2}+Q)^{2}}\cdot (|\fbm{x}_{ij}|^{2}\\
&-|\fbm{x}^{d}_{ij}|^{2})(\fbm{x}_{i}-\fbm{x}_{j})+\gamma\boldsymbol{\zeta}\left(\fbm{x}_{j}-\fbm{x}_{jd}\right)\\
\leq&\eta\bar{x}_{j}\fbm{1}+\gamma \boldsymbol{\zeta}\left(\fbm{x}_{j}-\fbm{x}_{jd}\right)\textrm{,}
\end{align*}
because $|\fbm{x}_{ij}|^{2}-|\fbm{x}^{d}_{ij}|^{2}=(|\fbm{x}_{ij}|+|x^{d}_{ij}|)(|\fbm{x}_{ij}|-|\fbm{x}^{d}_{ij}|)\leq (2r+n\overline{d}_{ji}\sqrt{2p\psi(r)})|\fbm{x}_{ij}-\fbm{x}^{d}_{ij}|\leq (2r+n\overline{d}_{ji}\sqrt{2p\psi(r)})\sqrt{n}\bar{x}_{j}$, where $\eta=\frac{4(r^{2}+Q)^{2}(2r+n\overline{d}_{ji}\sqrt{2p\psi(r)})\sqrt{n}r}{Q^{2}\Delta^{2}}$ and $\gamma=\frac{2(r^{2}+Q)}{\Delta^{2}}$.
\end{proof}

The above analysis is summarized in the following theorem:
\begin{theorem}
Given a single-integrator MAS with time-varying delays that satisfies assumption~\ref{A1}, the controls~\eqref{equ2} maintain the initial connectivity of the MAS and coordinate it if the damping gains $\fbm{K}_{i}$ are selected to satisfy the conditions~\eqref{equ5}.
\end{theorem}

Because the single-integrator MAS~\eqref{equ1} in closed loop with the controllers~\eqref{equ2} can be regarded as the double-integrator MAS $\dotbm{x}_{i}=\fbm{v}_{i}$, $\dotbm{v}_{i}=\bm{\tau}_{i}$ with the state-dependent proportional plus damping controls $\bm{\tau}_{i}=-\sum_{j\in\mathcal{N}_{i}(0)}\nabla_{i}\psi(|\fbm{x}^{d}_{ij}|)-\fbm{K}_{i}\dotbm{x}_{i}$, the above analysis readily applies to double-integrator MAS, whose connectivity-preserving coordination the controls $\bm{\tau}_{i}$ can enforce. 

In EL networks with time-varying delays, the agent dynamics are:
\begin{equation}\label{equ6}
\fbm{M}_{i}(\fbm{x}_{i})\ddotbm{x}_{i}+\fbm{C}_{i}(\fbm{x}_{i},\dotbm{x}_{i})\dotbm{x}_{i}+\fbm{g}_{i}(\fbm{x}_{i})=\bm{\tau}_{i}\textrm{,}
\end{equation} 
with $\fbm{x}_{i}$, $\dotbm{x}_{i}$ and $\ddotbm{x}_{i}$ the joint positions, velocities and accelerations; $\fbm{M}_{i}(\fbm{x}_{i})$ and $\fbm{C}_{i}(\fbm{x}_{i},\dotbm{x}_{i})$ the matrices of inertia and of Coriolis and centrifugal effects; and $\fbm{g}_{i}(\fbm{x}_{i})$ and $\bm{\tau}_{i}$ the gravity and control torques. The dynamics~\eqref{equ6} have the following properties:
\begin{enumerate}
\item[1.]
If agent $i$ has only revolute joints, its inertia matrix is symmetric, positive definite and uniformly bounded by $0<\lambda_{i1}\fbm{I}\preceq\fbm{M}_{i}(\fbm{x}_{i})\preceq\lambda_{i2}\fbm{I}<\infty$;
\item[2.]
$\dotbm{M}_{i}(\fbm{x}_{i})-2\fbm{C}_{i}(\fbm{x}_{i},\dotbm{x}_{i})$ is skew-symmetric.
\end{enumerate}

Taking the agent controls:
\begin{equation}\label{equ7}
\bm{\tau}_{i}=-p\sum_{j\in\mathcal{N}_{i}}\nabla_{i}\psi(|\fbm{x}^{d}_{ij}|)-\fbm{K}_{i}\dotbm{x}_{i}+\fbm{g}_{i}(\fbm{x}_{i})\textrm{,}
\end{equation}
the coordination with connectivity preservation of the EL network can be studied using the following Lyapunov function:
\begin{equation}\label{equ8}
V=\frac{1}{2}\sum^{N}_{i=1}\dottbm{x}_{i}\fbm{M}_{i}(\fbm{x}_{i})\dotbm{x}_{i}+\frac{p}{2}\sum^{N}_{i=1}\sum_{j\in\mathcal{N}^{0}_{i}}\psi(|\fbm{x}_{ij}|)\textrm{,}
\end{equation}
with $p$ selected to satisfy:
\begin{equation}\label{equ9}
\begin{aligned}
V(0)=&\frac{1}{2}\sum^{N}_{i=1}\dottbm{x}_{i}(0)\fbm{M}_{i}(\fbm{x}_{i}(0))\dotbm{x}_{i}(0)\\
&+\frac{p}{2}\sum^{N}_{i=1}\sum_{j\in\mathcal{N}^{0}_{i}}\psi(|\fbm{x}_{ij}(0)|)<p\psi(r)\textrm{.}
\end{aligned}
\end{equation} 
\begin{proposition}
There exists a (sufficiently large) $p$ that satisfies condition~\eqref{equ9} for any initial EL network configuration.
\end{proposition}
\begin{proof}
From assumption~\ref{A1}, $|\fbm{x}_{ij}(0)|\leq r-\epsilon$, $\forall (i,j)\in\mathcal{E}(0)$. The first two properties of $\psi(\cdot)$ imply that $\frac{p}{2}\sum^{N}_{i=1}\sum_{j\in\mathcal{N}_{i}(0)}\psi(|\fbm{x}_{ij}(0)|)\leq \frac{pN(N-1)}{2}\psi(r-\epsilon)$. Then, given a bound $KE(0)$ on the initial kinetic energy of the EL network $\frac{1}{2}\sum^{N}_{i=1}\dottbm{x}_{i}(0)\fbm{M}_{i}(\fbm{x}_{i}(0))\dotbm{x}_{i}(0)\leq KE(0)$, condition~\eqref{equ9} holds if $p\left[\psi(r)-\frac{N(N-1)}{2}\psi(r-\epsilon)\right]>KE(0)$ holds, i.e., if $p$ is selected to satisfy  $p>\frac{2KE(0)}{2\psi(r)-N(N-1)\psi(r-\epsilon)}$.
\end{proof} 

A sufficient condition for connectivity maintenance is that the set $\{(\fbm{x},\dotbm{x})|\ V(t)\leq V(0)<p\psi(r)\}$ be invariant, or that the derivative of the Lyapunov function~\eqref{equ8} be non-positive. This derivative is:
\begin{align*}
\dot{V}=&\frac{1}{2}\sum^{N}_{i=1}\dottbm{x}_{i}\dotbm{M}_{i}(\fbm{x}_{i})\dotbm{x}_{i}+\sum^{N}_{i=1}\dottbm{x}_{i}\fbm{M}_{i}(\fbm{x}_{i})\ddotbm{x}_{i}\\
&+\frac{p}{2}\sum^{N}_{i=1}\sum_{j\in\mathcal{N}^{0}_{i}}\left[\dottbm{x}_{i}\nabla_{i}\psi(|\fbm{x}_{ij}|)+\dottbm{x}_{j}\nabla_{j}\psi(|\fbm{x}_{ij}|)\right]\\
=&\frac{1}{2}\sum^{N}_{i=1}\dottbm{x}_{i}\left[\dotbm{M}_{i}(\fbm{x}_{i})-2\fbm{C}_{i}(\fbm{x}_{i},\dotbm{x}_{i})\right]\dotbm{x}_{i}-\sum^{N}_{i=1}\dottbm{x}_{i}\fbm{K}_{i}\dotbm{x}_{i}\\
&+p\sum^{N}_{i=1}\sum_{j\in\mathcal{N}^{0}_{i}}\dottbm{x}_{i}\left[\nabla_{i}\psi(|\fbm{x}_{ij}|)-\nabla_{i}\psi(|\fbm{x}^{d}_{ij}|)\right]\textrm{.}
\end{align*}
An upper-bound on $\dot{V}$ can be derived based on the skew-symmetry of the EL dynamics:
\begin{align*}
\dot{V}=&p\sum^{N}_{i=1}\sum_{j\in\mathcal{N}^{0}_{i}}\dottbm{x}_{i}\left[\nabla_{i}\psi(|\fbm{x}_{ij}|)-\nabla_{i}\psi(|\fbm{x}^{d}_{ij}|)\right]-\sum^{N}_{i=1}\dottbm{x}_{i}\fbm{K}_{i}\dotbm{x}_{i}\\
\leq&p\sum^{N}_{i=1}\sum_{j\in\mathcal{N}^{0}_{i}}\gamma\boldsymbol{\zeta}\left(\dotbm{x}_{i}\right)^\mathsf{T}\int^{t}_{t-d_{ji}(t)}\boldsymbol{\zeta}\left(\dotbm{x}_{j}(\theta)\right)d\theta\\
&+p\sum^{N}_{i=1}\sum_{j\in\mathcal{N}^{0}_{i}}\eta\bar{x}_{j}\boldsymbol{\zeta}\left(\dotbm{x}_{i}\right)^\mathsf{T}\fbm{1}-\sum^{N}_{i=1}\dottbm{x}_{i}\fbm{K}_{i}\dotbm{x}_{i}\textrm{,}
\end{align*}
and yields an upper-bound on the integral of $\dot{V}$ from $0$ to $t$:
\begin{align*}
V(t)\leq&-\sum^{N}_{i=1}k_{i}\|\dotbm{x}_{i}\|^{2}_{2}+\sum^{N}_{i=1}\sum_{j\in\mathcal{N}^{0}_{i}}\Bigg(\frac{\alpha_{ij}p(\gamma+\eta)}{2}\|\dotbm{x}_{i}\|^{2}_{2}\\
&+\frac{p(\gamma+n\eta)\overline{d}^{2}_{ji}}{2\alpha_{ij}}\|\dotbm{x}_{j}\|^{2}_{2}\Bigg)+V(0)\textrm{.}
\end{align*}

An analysis similar to that of a single-integrator MAS shows that the controllers~\eqref{equ7} maintain the initial connectivity and coordinate the EL network if the damping gains $\fbm{K}_{i}$ satisfy:
\begin{equation}\label{equ10}
k_{i}>\frac{p}{2}\sum_{j\in\mathcal{N}^{0}_{i}}\left(\alpha_{ij}(\gamma+\eta)+\frac{(\gamma+n\eta)\overline{d}^{2}_{ij}}{\alpha_{ji}}\right)\textrm{.}
\end{equation}

\begin{proposition}
The feasibility of the control~\eqref{equ7} derived from the potential $\psi(|\fbm{x}_{ij}|)=\frac{|\fbm{x}_{ij}|^{2}}{r^{2}-|\fbm{x}_{ij}|^{2}+Q}$ depends on the initial state of the EL network.
\end{proposition}
\begin{proof}
From previous propositions: 1) $Q<\frac{2r^{2}*(r^{2}-(r-\epsilon)^{2})}{N(N-1)(r-\epsilon)^{2}-2r^{2}}$ guarantees $\frac{N(N-1)}{2}\psi(r-\epsilon)<\psi(r)$ when $N(N-1)>\frac{2r^{2}}{(r-\epsilon)^{2}}$; 2) $\nabla_{i}\psi(|\fbm{x}^{d}_{ij}|)$ exist if $Q\geq 2pn^{2}\overline{d}^{2}_{ji}\psi(r)+2nr\overline{d}_{ji}\sqrt{2p\psi(r)}+\Delta$ for some positive $\Delta$; and 3) $p>\frac{2KE(0)}{2\psi(r)-N(N-1)\psi(r-\epsilon)}$ satisfies condition~\eqref{equ9}. The first condition is an upper bound on $Q$ that induces an upper bound on $p$ by the second condition. The third condition is a lower bound on $p$ that depends on the initial state of the EL network. Hence, a feasible $Q$ and $p$ pair may not exist for some initial system states. 
\end{proof}

Though potentially unfeasible, the control~\eqref{equ7} is feasible for any EL network initially at rest. The following theorem summarizes the connectivity-preserving coordination of EL networks based on the potential $\psi(|\fbm{x}_{ij}|)=\frac{|\fbm{x}_{ij}|^{2}}{r^{2}-|\fbm{x}_{ij}|^{2}+Q}$.

\begin{theorem}
Given an EL network~\eqref{equ6} with time-varying communication delays that satisfies assumption~\ref{A1}, the control torques~\eqref{equ7} synchronize the MAS and preserve its connectivity if the parameter $p$ and the damping gains $\fbm{K}_{i}$ can be selected to satisfy conditions~\eqref{equ9} and~\eqref{equ10}.
\end{theorem}

\section{Simulations}\label{Sec: simulations}
This section illustrates the designed control strategies on two simulated MAS-s, a single-integrator MAS and an EL network. The simulated networks have $N=5$ agents with broadcasting radius $r=1$~m, buffer width $\epsilon=0.4$~m and maximum communication delays $\overline{d}_{ij}=0.1$~s, $\forall (i,j)\in\mathcal{E}(0)$. 

The simulated $1$-dimensional single-integrator agents have initial positions: $x_{1}=1$~m, $x_{2}=1.5$~m, $x_{3}=2.1$~m, $x_{4}=2.7$~m and $x_{5}=3.2$~m. The control parameters are selected as: $Q=0.2$ to satisfy the second property of $\psi(\cdot)$; and $K_{1}=K_{5}=30$ and $K_{2}=K_{3}=K_{4}=60$ to satisfy condition~\eqref{equ5}. \fig{fig1} shows that the controls~\eqref{equ2} synchronize the $5$ agents while preserving their initial connections.
\begin{figure}[!hbt]
\centering
\includegraphics[width=\columnwidth , height=6cm]{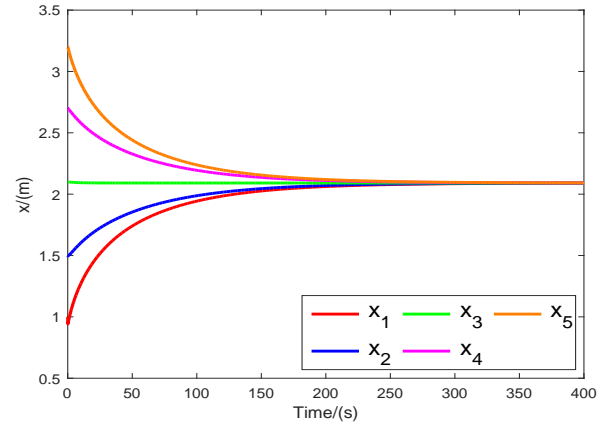}
\caption{Coordination of a simulated single-integrator MAS.}
\label{fig1}
\end{figure}

The simulated EL agents are $2$-DOF robots with link masses $m_{i1}=m_{i2}=0.5$~kg and lengths $l_{i1}=l_{i2}=1$~m. Their initial joint positions are: $\fbm{q}_{1}=\begin{bmatrix}\frac{\pi}{12}, -\frac{5\pi}{12}\end{bmatrix}^\mathsf{T}$, $\fbm{q}_{2}=\begin{bmatrix}\frac{\pi}{6}, -\frac{5\pi}{12}\end{bmatrix}^\mathsf{T}$, $\fbm{q}_{3}=\begin{bmatrix}\frac{5\pi}{24}, -\frac{7\pi}{24}\end{bmatrix}^\mathsf{T}$, $\fbm{q}_{4}=\begin{bmatrix}\frac{\pi}{4}, -\frac{5\pi}{24}\end{bmatrix}^\mathsf{T}$ and $\fbm{q}_{5}=\begin{bmatrix}\frac{9\pi}{24}, -\frac{5\pi}{12}\end{bmatrix}^\mathsf{T}$. For zero initial velocities, the selection $Q=0.2$, $p=0.01$, $\fbm{K}_{1}=360\fbm{I}$, $\fbm{K}_{2}=\fbm{K}_{4}=\fbm{K}_{5}=720\fbm{I}$ and $\fbm{K}_{3}=1080\fbm{I}$ satisfies conditions~\eqref{equ9} and~\eqref{equ10}. \fig{fig2} validates that the controls~\eqref{equ7} coordinate the EL network and maintain its connectivity.
\begin{figure}[!hbt]
\centering
\includegraphics[width=\columnwidth , height=6cm]{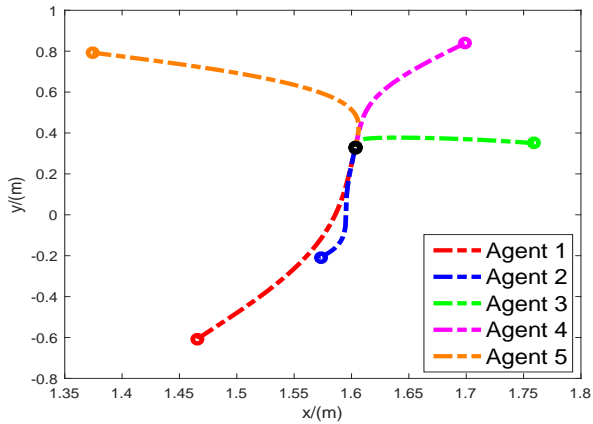}
\caption{Coordination of a simulated EL network.}
\label{fig2}
\end{figure}

\section{Conclusions}\label{Sec: conclusions}
This paper has designed a distributed control strategy for the coordination with local connectivity preservation of single-integrator MAS-s, and of EL networks, in the presence of time-varying delays. The strategy is a proportional plus damping injection controller. The proportional term minimizes the energy of the inter-agent communications as measured by a potential function designed to maintain agents within broadcasting distance of all their initial neighbours. The damping term suppresses the destabilizing disturbances introduced in the proportional control by the delayed position signals. Sufficient damping is injected indirectly, through dynamic filters, in single-integrator MAS-s, and directly in EL networks. Lyapunov stability analysis reveals the relationship between the damping gains and the selected potential, and leads to sufficient conditions on the control gains for connectivity-preserving coordination of the two types of networks. 


\bibliography{bibi}
\end{document}